\documentclass{amsart}   
\usepackage[margin=1in]{geometry}
\usepackage{enumerate}
\usepackage{verbatim}

\usepackage{amscd}     
\usepackage{amsthm}   
\usepackage{amssymb}
\usepackage{amsmath}
\usepackage{amsxtra}
\usepackage{mathabx}
\usepackage{color}
\usepackage[all]{xy}
\newtheorem{theorem}{Theorem}[section]
\newtheorem{lemma}[theorem]{Lemma}

\theoremstyle{definition}
\newtheorem{definition}[theorem]{Definition}
\newtheorem{example}[theorem]{Example}

\newtheorem{proposition}[theorem]{Proposition}

\theoremstyle{remark}
\newtheorem{remark}[theorem]{Remark}

\numberwithin{equation}{section}

\theoremstyle{remark}

\newcommand\strongof[1]{{#1}^+}
\newcommand \partitionof[1]{\widetilde{#1}}
\newcommand\reverse[1]{{#1}^*}

\newcommand{\Qsym}{\textsl{QSym}}
\newcommand{\xx}{\mathbf{x}}

\newcommand{\rsqschur}{\mathcal{RS}}
\newcommand{\csqschur}{\mathcal{CS}}

\newcommand{\kone}{\mathcal{K}_1}
\newcommand{\ktwo}{\mathcal{K}_2}

\DeclareMathOperator{\cont}{cont}

\definecolor{gr90}{gray}{0.90}

\definecolor{gr75}{gray}{0.75}

\definecolor{gyblue}{cmyk}{0,0.5,0,0}

\newsavebox{\myonesquare}
\savebox{\myonesquare}{\textcolor{gr75}{\rule{12.5pt}{12.5pt}}}
\newsavebox{\mytwosquare}
\savebox{\mytwosquare}{\textcolor{gr75}{\rule{25pt}{12.5pt}}}
\newsavebox{\mythreesquare}
\savebox{\mythreesquare}{\textcolor{gr75}{\rule{37.5pt}{12.5pt}}}
\newsavebox{\myfoursquare}
\savebox{\myfoursquare}{\textcolor{gr75}{\rule{50pt}{12.5pt}}}

\newsavebox{\othertwosquare}
\savebox{\othertwosquare}{\textcolor{gr75}{\rule{25pt}{12.5pt}}}

\newcommand{\sqone}{\usebox{\myonesquare}}
\newcommand{\sqtwo}{\usebox{\mytwosquare}}
\newcommand{\sqthree}{\usebox{\mythreesquare}}

\newlength{\cellsize} 
\newsavebox{\cell}

\sbox{\cell}{\begin{picture}(18,18) \put(0,0){\line(1,0){18}}
\put(0,0){\line(0,1){18}} \put(18,0){\line(0,1){18}}
\put(0,18){\line(1,0){18}}
\end{picture}}

\newcommand\cellify[1]{\def\thearg{#1}\def\nothing{}%
\ifx\thearg\nothing \vrule width0pt height\cellsize depth0pt\else
\hbox to 0pt{\usebox{\cell} \hss}\fi%
\vbox to \cellsize{ \vss \hbox to \cellsize{\hss$#1$\hss} \vss}}

\newcommand\tableau[1]{\vtop{\let\\\cr
\baselineskip -16000pt \lineskiplimit 16000pt \lineskip 0pt
\ialign{&\cellify{##}\cr#1\crcr}}}

\savebox3{%
\begin{picture}(30,30)
\put(0,0){\line(1,0){30}}
\put(0,0){\line(0,1){30}}
\put(30,0){\line(0,1){30}}
\put(0,30){\line(1,0){30}}
\end{picture}}
\newcommand\bigcellify[1]{\def\thearg{#1}\def\nothing{}%
\ifx\thearg\nothing
\vrule width0pt height\cellsize depth0pt\else
\hbox to 0pt{\usebox3\hss}\fi%
\vbox to 30\unitlength{
\vss
\hbox to 30\unitlength{\hss$#1$\hss}
\vss}}
\newcommand\Tableau[1]{\vtop{\let\\=\cr
\setlength\baselineskip{-16000pt}
\setlength\lineskiplimit{16000pt}
\setlength\lineskip{0pt}
\halign{&\bigcellify{##}\cr#1\crcr}}}

\newcommand\bas[1]{\omit \vbox to \cellsize{ \vss \hbox to \cellsize{\hss$#1$\hss} \vss}}

\begin{document}

\title[A Littlewood-Richardson Type Rule]{A Littlewood-Richardson Type Rule for Row-Strict Quasisymmetric Schur Functions}


\author[J. Ferreira]{Jeffrey Ferreira}
\thanks{Supported in part by NSF VIGRE grant DMS0636297.}
\address{Department of Mathematics, University of California at Davis}
\email{jferreira@math.ucdavis.edu}
 



\begin{abstract}
We give a Littlewood-Richardson type rule for expanding the product of a row-strict quasisymmetric Schur function and a symmetric Schur function in terms of row-strict quasisymmetric Schur functions. This expansion follows from several new properties of an insertion algorithm defined by Mason and Remmel (2010) which inserts a positive integer into a row-strict composition tableau. 
\end{abstract}

\maketitle


\section{Introduction}

In \cite{Haglund2008Quasisymmetric}, the authors define a new basis of the algebra $\Qsym$ of quasisymmetric functions called column-strict quasisymmetric Schur functions, denoted $\csqschur_\alpha$, where $\alpha$ is a sequence of positive integers called a strong composition.  Over a fixed number of variables, the functions $\csqschur_\alpha$ are defined to be a certain positive integral sum of Demazure atoms.  Demazure atoms are related to Demazure characters and arise as specializations of nonsymmetric Macdonald polynomials when $q=t=0$ \cite{Mason2009An-explicit}. Demazure atoms were studied in \cite{Lascoux1990Keys} where the authors called them ``standard bases." The functions $\csqschur_\alpha$ over a finite number of variables were shown in \cite{Lauve2010QSym} to give a basis of the coinvariant space of quasisymmetric polynomials, thus proving a conjecture of Bergeron and Reutenauer in \cite{Bergeron_Reutenauer}.

In \cite{Haglund2009Refinements} the authors give a Littlewood-Richardson type rule for expanding the product $\csqschur_\alpha s_\lambda$, where $s_\lambda$ is the symmetric Schur function, as a nonnegative integral sum of the functions $\csqschur_\beta$.  This rule relied on a definition for $\csqschur_\alpha$ as the generating function of column-strict composition tableaux, which are certain fillings with positive integers of strong composition shape $\alpha$. These column-strict composition tableaux are defined by imposing three relations among certain sets of entries in the fillings of $\alpha$.  The proof of the Littlewood-Richardson type rule in \cite{Haglund2009Refinements} utilized an analogue of Schensted insertion on tableaux, which is an algorithm in classical symmetric function theory which inserts a positive integer $b$ into a tableau $T$. The results in this paper were inspired by \cite{Haglund2009Refinements}.

In \cite{Mason2010A-Dual-Basis}, the authors provide a row-strict analogue of column-strict composition tableaux; specifically they interchange the roles of weak and strict in each of the three relations mentioned above.  One of these relations requires the fillings to decrease strictly across each row, thus the name row-strict composition tableaux.  Also contained in \cite{Mason2010A-Dual-Basis} is an insertion algorithm which inserts a positive integer $b$ into a row-strict composition tableau, producing a new row-strict composition tableau.

This article establishes several new properties of the insertion algorithm given in \cite{Mason2010A-Dual-Basis}.  These properties lead directly to a Littlewood-Richardson type rule for expanding the product $\rsqschur_\alpha s_\lambda$ as a nonnegative integral sum of the function $\rsqschur_\beta$.  The combinatorics of this rule share many similarities with the classical Littlewood-Richardson rule for multiplying two Schur functions, see \cite{Fulton1997Young-tableaux} for an example.

\section{Definitions}

\subsection{Compositions and reverse lattice words}

A \emph{strong composition} $\alpha=(\alpha_1, \ldots, \alpha_k)$ with $k$ parts is a sequence of positive integers, a \emph{weak composition} $\gamma=(\gamma_1, \ldots, \gamma_k)$ is a sequence of nonnegative integers, and a \emph{partition} $\lambda=(\lambda_1, \ldots, \lambda_k)$ is a weakly decreasing sequence of positive integers.  Let \begin{math} \reverse{\lambda}:=(\lambda_k, \lambda_{k-1}, \ldots, \lambda_1) \end{math} be the \emph{reverse of} $\lambda$, and let $\lambda^t$ denote the usual \emph{transpose of} $\lambda$. Denote by $\partitionof{\alpha}$ the unique partition obtained by placing the parts of $\alpha$ in weakly decreasing order.  Denote by $\strongof{\gamma}$ the unique strong composition obtained by removing the zero parts of $\gamma$. For any sequence $\beta=(\beta_1, \ldots, \beta_s)$ let $\ell(\beta):=s$ be the \emph{length of} $\beta$. For $\gamma$ and $\beta$ arbitrary (possibly weak) compositions of the same length $s$ we say $\gamma$ is \emph{contained in} $\beta$, denoted $\gamma \subseteq \beta$, if $\gamma_i \leq \beta_i$ for all $1 \leq i \leq s$.

A finite sequence $w=w_1w_2\cdots w_n$ of positive integers with largest part $m$ is called a \emph{reverse lattice word} if in every prefix of $w$ there are at least as many $i$'s as $(i-1)$'s for each $1< i \leq m$.  The \emph{content} of a word $w$ is the sequence $\cont(w)=(\cont(w)_1, \ldots, \cont(w)_m)$ with the property that $\cont(w)_i$ equals the number of times $i$ appears in $w$.  A reverse lattice word is called \emph{regular} if $\cont(w)_1 \neq 0$.  Note that if $w$ is a regular reverse lattice word, then $\cont(w)=\reverse{\lambda}$ for some partition $\lambda$.

\subsection{Diagrams and fillings}

To any sequence $\alpha$ of nonnegative integers we may associate a \emph{diagram}, also denoted $\alpha$, of left justified boxes with $\alpha_i$ boxes in the $i$th row from the top.  In the case $\alpha=\lambda$ is a partition, the diagram of $\lambda$ is the usual Ferrers diagram in English notation.  Given a diagram $\alpha$, let $(i,j)$ denote the box in the $i$th row and $j$th column.

Given two sequences $\gamma$ and $\alpha$ of the same length $s$ such that $\gamma \subseteq \alpha$, define the \emph{skew diagram} $\alpha/\gamma$ to be the array of boxes that are in $\alpha$ and not in $\gamma$.  The boxes in $\gamma$ are called the \emph{skewed boxes.}  For each skew diagram in this article an extra column, called the \emph{$0$th column}, with $s$ boxes will be added strictly to the left of each existing column.

A \emph{filling} $U$ of a diagram $\alpha$ is an assignment of positive integers to the boxes of $\alpha$.  Given a filling $U$ of $\alpha$, let $U(i,j)$ be the entry in the box $(i,j)$.  A \emph{reverse row-strict tableau}, or RRST, $T$ is a filling of partition shape $\lambda$ such that each row strictly decreases when read left to right and each column weakly decreases when read top to bottom.  If $\lambda$ is a partition with $\lambda_1=m$, then let $T_\lambda$ be the tableau of shape $\lambda$ which has the entire $i$th column filled with the entry $(m+1-i)$ for all $1\leq i \leq m$.

A filling $U$ of a skew diagram $\alpha/\gamma$ is an assignment of positive integers to the boxes that are in $\alpha$ and not in $\gamma$. We follow the convention that each box in the $0$th column and each skewed box is assigned a virtual $\infty$ symbol.  With this convention, an entry $U(i,j)$ may equal $\infty$. Given two boxes filled with $\infty$, if they are in the same row we define these entries to strictly decrease left to right, while two such boxes in the same column are defined to be equal. 

The \emph{column reading order} of a (possibly skew) diagram is the total order $<_{col}$ on its boxes where $(i,j) <_{col} (i^\prime, j^\prime)$ if $j<j^\prime$ or ($j=j^\prime$ and $i>i^\prime$).  This is the total order obtained by reading the boxes from bottom to top in each column, starting with the left-most column and working rightwards. If $\alpha$ is a diagram with $k$ rows and longest row length $m$, it will occasionally be convenient to define this order on all cells $(i,j)$, where $0\leq i \leq k$ and $1 \leq j \leq m+2$, regardless of whether the cell $(i,j)$ is a box in $\alpha$.  The \emph{column reading word} of a (possibly skew) filling $U$ is the sequence of integers $w_{col}(U)$ obtained by reading the entries of $U$ in column reading order, where we ignore entries from skewed boxes and entries in the $0$th column.

The following definition first appeared in \cite{Mason2010A-Dual-Basis}.

\begin{definition}{\label{def_RCT}}
Let $\alpha$ be a strong composition with $k$ parts and largest part size $m$.  A {\it{row-strict composition tableau}} (RCT) $U$ is a filling of the diagram $\alpha$ such that
\begin{enumerate}
\item The first column is weakly increasing when read top to bottom.
\item Each row strictly decreases when read left to right.
\item \textbf{Triple Rule:} Supplement $U$ with zeros added to the end of each row so that the resulting filling $\hat{U}$ is of rectangular shape $k \times m$.  Then for $1 \leq i_1 < i_2 \leq k$ and $2 \leq j \leq m$,
\begin{equation*}
\left( \hat{U}(i_2,j) \neq 0 \text{ and } \hat{U}(i_2,j)>\hat{U}(i_1,j) \right)\Rightarrow \hat{U}(i_2,j) \geq \hat{U}(i_1, j-1).
\end{equation*}
\end{enumerate}
\end{definition}

If we let $\hat{U}(i_2,j)=b$, $\hat{U}(i_1,j)=a$, and $\hat{U}(i_1, j-1)=c$, then the Triple Rule ($b\neq 0$ and $b>a$ implies $b \geq c$) can be pictured as

\[
\begin{array}{ccc} \vspace{6pt}
\tableau{   c & a   \\  & \bas{{\vdots}} }  
 \\
 \tableau{ & b}
\end{array}.
\]

In addition to the triples that satisfy Definition \ref{def_RCT}, we also have a notion of inversion triples. Inversion triples were originally introduced by Haglund, Haiman, and Loehr in \cite{Haglund2005A-combinatorial} to describe a combinatorial formula for symmetric, and later nonsymmetric, Macdonald polynomials.  In the present context inversion triples are defined as follows. Let $\gamma$ be a (possibly weak) composition and let $\beta$ be a strong composition with $\gamma \subseteq \beta$.  Let $U$ be some arbitrary filling of $\beta/\gamma$.  A \emph{Type A triple} is a triple of entries
\begin{displaymath}
U(i_1, j-1)=c, U(i_1, j)=a, U(i_2, j)=b
\end{displaymath}
in $U$ with $\beta_{i_1} \geq \beta_{i_2}$ for some rows $i_1 < i_2$ and some column $j> 0$.  A \emph{Type B triple} is a triple of entries
\begin{displaymath}
U(i_1, j)=b, U(i_2, j)=c, U(i_2, j+1)=a
\end{displaymath}
in $U$ with $\beta_{i_1}<\beta_{i_2}$ for some rows $i_1<i_2$ and some column $j\geq 0$. A triple of either type A or B is said to be an \emph{inversion triple} if either $b \leq a < c$ or $a < c \leq b$.  Note that triples of either type may involve skewed boxes or boxes in the $0$th column. Type A and Type B triples can be visualized as

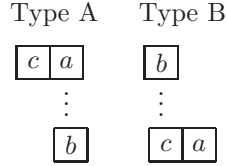
\[
\begin{array}{cc} \vspace{6pt}
\text{Type A} & \text{Type B} \\

\begin{array}{ccc} \vspace{6pt}
\tableau{   c & a   \\  & \bas{{\vdots}} }  
 \\
 \tableau{ & b}
\end{array}
&
\begin{array}{ccc} \vspace{6pt}
\tableau{   b &    \\  \bas{{\vdots}} & }  
 \\
 \tableau{c & a}
\end{array}

\end{array}.
\]

Central to the main theorem of this paper is the following definition.

\begin{definition}{\label{LRskewCT}}
Let $\beta$ and $\alpha$ be strong compositions. Let $\gamma$ be some (possibly weak) composition satisfying $\strongof{\gamma}=\alpha$ and $\gamma \subseteq \beta$. A {\it{Littlewood-Richardson skew row-strict composition tableau}} $S$, or {\it{LR skew RCT}}, of shape $\beta/\alpha$ is a filling of a diagram of skew shape $\beta/\gamma$ such that: 
\begin{enumerate}
\item Each row strictly decreases when read left to right.
\item Every Type A and Type B triple is an inversion triple.
\item The column reading word of $S$, $w_{col}(S)$, is a regular reverse lattice word.
\end{enumerate}
\end{definition}

Note that in Definition \ref{LRskewCT} the shape of an LR skew RCT is $\beta/\alpha$ although we refer to a filling of $\beta/\gamma$.

\begin{example}{\label{UandS}}
Below is a RCT, $U$, of shape $(1,3,2,2)$, and a LR skew RCT, $S$, of shape $(1,2,3,1,5,3)/(1,3,2,2)$ with $w_{col}(S)=4433421$.

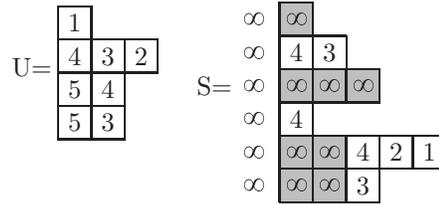
\begin{picture}(100,100)(-75,-10)
\put(100,50){U=}
\put(125,70){\tableau{ 1 \\
	      4 & 3 & 2 \\
	      5 & 4 \\
	      5 & 3
	      }}
\put(200, 40){S=}
\put(245,72){\sqone}
\put(245,36){\sqthree}
\put(245,0){\sqtwo}
\put(245,-18){\sqtwo}
\put(222,72){\tableau{ \bas{\infty} &\infty \\
    \bas{\infty} & 4& 3  \\
     \bas{\infty} & \infty & \infty & \infty  \\ 
    \bas{\infty} & 4  \\
    \bas{\infty} & \infty & \infty &4&2&1 \\ 
    \bas{\infty} & \infty & \infty  & 3
}}
\end{picture}
\end{example}

\subsection{Generating functons}

The content of any filling $U$ of partition or composition shape, denoted $\cont(U)$, is the content of its column reading word $w_{col}(U)$.  To any filling $U$ we may associate a monomial 
\begin{equation*}
\xx^U=\prod_{i \geq 1} x_i^{\cont(U)_i}.
\end{equation*}  

The algebra of symmetric functions $\Lambda$ has the Schur functions $s_\lambda$ as a basis, where $\lambda$ ranges over all partitions.  The Schur function $s_\lambda$ can be defined in a number of ways.  In this article it is advantageous to define $s_\lambda$ as the generating function of reverse row-strict tableaux of shape $\lambda^t$.  That is

\begin{displaymath}
s_\lambda = \sum \xx^T
\end{displaymath}
where the sum is over all reverse row-strict tableaux $T$ of shape $\lambda^t$. See \cite{Stanley1999EC2} for many of the properties of $s_\lambda$.

The generating function of row-strict composition tableaux of shape $\alpha$ are denoted $\rsqschur_\alpha$.  That is
\begin{displaymath}
\rsqschur_\alpha=\sum \xx^U
\end{displaymath}
where the sum is over all row-strict composition tableaux $U$ of shape $\alpha$.  The generating functions $\rsqschur_\alpha$ are called \emph{row-strict quasisymmetric Schur functions} and were originally defined in \cite{Mason2010A-Dual-Basis}.  In \cite{Mason2010A-Dual-Basis} the authors show $\rsqschur_\alpha$ are indeed quasisymmetric, and furthermore the collection of all $\rsqschur_\alpha$, as $\alpha$ ranges over all strong compositions, forms a basis of the algebra $\Qsym$ of quasisymmetric functions. The authors also show that the Schur function $s_\lambda$ decomposes into a positive sum of row-strict quasisymmetric Schur functions indexed by compositions that rearrange the transpose of $\lambda$.  Specifically,
\begin{equation*}
s_\lambda=\sum_{\partitionof{\alpha}=\lambda^t} \rsqschur_\alpha .
\end{equation*}

\section{Insertion algorithms}
\label{sec:algo}

Define a \emph{two-line array} $A$ by letting
\begin{equation*}
A=\left( \begin{matrix} i_1 & i_2 & \cdots & i_n \\
				   j_1 & j_2 & \cdots & j_n 	 \end{matrix} \right)
\end{equation*}
where $i_r, j_r$ are positive integers for $1 \leq r \leq n$, (a) $i_1 \geq i_2 \geq \cdots \geq i_n$, and (b) if $i_r=i_s$ and $r\leq s$ then $j_r \leq j_s$. Denote by $\widehat{A}$ the upper sequence $i_1, i_2, \ldots, i_n$ and denote by $\widecheck{A}$ the lower sequence $j_1, j_2, \ldots, j_n$.

The classical Robinson-Schensted-Knuth (RSK) correspondence gives a bijection between two-line arrays $A$ and pairs of (reverse row-strict) tableaux $(P,Q)$ of the same shape \cite{Fulton1997Young-tableaux}.  The basic operation of RSK is Schensted insertion on tableaux, which is an algorithm that inserts a positive integer $b$ into a tableau $T$ to produce a new tableau $T^\prime$.  In our setting, Schensted insertion is restated as

\begin{definition}
Given a tableau $T$ and $b$ a positive integer one can obtain $T^\prime :=b \rightarrow T$ by inserting $b$ as follows:
\begin{enumerate}
\item Let $\tilde{b}$ be the largest entry less than or equal to $b$ in the first row of $T$.  If no such $\tilde{b}$ exists, simply place $b$ at the end of the first row.
\item If $\tilde{b}$ does exists, replace (bump) $\tilde{b}$ with $b$ and proceed to insert $\tilde{b}$ into the second row using the method just described.
\end{enumerate}
\end{definition}

The RSK correspondence is the bijection obtained by inserting $\widecheck{A}$ in the empty tableau $\emptyset$ to obtain a tableau $P$ called the \emph{insertion tableau}, while simultaneously placing $\widehat{A}$ in the corresponding new boxes to obtain a tableau $Q$ called the \emph{recording tableau}.

The authors in \cite{Mason2010A-Dual-Basis} provide an analogous insertion algorithm on row-strict composition tableaux.

\begin{definition}{\label{RCT Insertion}}{({\it{RCT Insertion}})}
Let $U$ be a RCT with longest row of length $m$, and let $b$ be a positive integer.  One can obtain $U^\prime :=U \leftarrow b$ by inserting $b$ as follows. Scan the entries of $U$ in reverse column reading order, that is top to bottom in each column starting with the right-most column and working leftwards, starting with column $m+1$ subject to the conditions:
\begin{enumerate}
\item In column $m+1$, if the current position is at the end of a row of length $m$, and $b$ is strictly less than the last entry in that row, then place $b$ in this empty position and stop.  If no such position is found, continue scanning at the top of column $m$.  
\item 
\begin{enumerate}
\item Inductively, suppose some entry $b_j$ begins scanning at the top of column $j$.  In column $j$, if the current position is empty and is at the end of a row of length $j-1$, and $b_j$ is strictly less than the last entry in that row, then place $b_j$ in this empty position and stop.  
\item If a position in column $j$ is nonempty and contains $\tilde{b}_j \leq b_j$ such that $b_j$ is strictly less than the entry immediately to the left of $\tilde{b}_j$, then $b_j$ bumps $\tilde{b}_j$ and continue scanning column $j$ with the entry $\tilde{b}_j$, bumping whenever possible. After scanning the last entry in column $j$, begin scanning column $j-1$.
\end{enumerate}
\item If an entry $b_1$ is bumped into the first column, then place $b_1$ in a new row that appears after the lowest entry in the first column that is weakly less than $b_1$.
\end{enumerate}
\end{definition}

In \cite{Mason2010A-Dual-Basis} the authors show $U^\prime=U \leftarrow b$ is a row-strict composition tableau.  The algorithm of inserting $b$ into $U$ determines a set of boxes in $U^\prime$ called the \emph{insertion path of $b$} and denoted $I(b)$, which is the set of boxes in $U^\prime$ which contain an entry bumped during the algorithm.  Not that if some entry $b_j$ bumps an entry $\tilde{b}_j$ then $b_j \geq \tilde{b}_j$; thus the sequence of entries bumped during the algorithm is weakly decreasing. We call the row in $U^\prime$ in which the new box is ultimately added the \emph{row augmented by insertion.} If the new box has coordinates $(i,1)$, then for each $r>i$, row $r$ of $U^\prime$ is said to be the \emph{corresponding row} of row $(r-1)$ of $U$. 

\begin{example} The figure below gives an example of the RCT insertion algorithm, where row $4$ is the row augmented by insertion.  The italicized entries indicated the insertion path $I(4)$.

\begin{picture}(100,100)(-90,-10)
\put(125,60){\tableau{ 1 \\
	       3 \\
	      4 & 3 & 2 \\
	      5 & 4 & 2\\
	      5 & 4
	      }}
\put(190, 20){$\leftarrow$}
\put(210, 20){$4$}
\put(230, 20){$=$}

\put(250,60){\tableau{
	1 \\
	3 \\
	4 & 3 & 2 \\
	{\bf{\it{4}}} \\
	5 & {\bf{\it{4}}} & 2 \\
	5 & {\bf{\it{4}}} }}

\end{picture}

\end{example}

We establish several new lemmas concerning RCT insertion that are instrumental in proving the main theorem of this paper in Section \ref{sec:LRrule}.

\begin{lemma}{\label{rows}}
Let $U$ be a RCT and let $b$ be a positive integer.  Then each row of $U^\prime = U \leftarrow b$ contains at most one box from $I(b)$.
\end{lemma}

\begin{proof}
Suppose for a contradiction that some row $i$ in $U^\prime$ contains at least two boxes from the insertion path of $b$.  Consider two of these boxes, say in columns $j$ and $j^\prime$ such that (without loss of generality) $j^\prime<j$.  Let $U^\prime(i,j)=b_1$ and $U^\prime(i, j^\prime)=b_2$.  Since $b_1$ was bumped earlier in the algorithm than $b_2$, we must have $b_2 \leq b_1$. Since $b_1$ and $b_2$ are in the same row, and $b_2$ appears to the left of $b_1$,  this contradicts row strictness of $U^\prime$.
\end{proof}

\begin{lemma}{\label{lastrow}} Let $U$ be a RCT and let $b$ be a positive integer.  Let $U^\prime = U \leftarrow b$ with row $i$ of $U^\prime$ the row augmented by insertion.  Then for all rows $r >i$ of $U^\prime$, the length of row $r$ is not equal to the length of row $i$.
\end{lemma}

\begin{proof}
Suppose for a contradiction that this is not the case.  Then there exists a row $r$ of $U^\prime$, $r>i$, whose length is equal to the length of row $i$.  Call this length $j$.  Since $i$ is the row in which the new cell was added then row $r$ in $U^\prime$ is the same as row $r$ in $U$, except in the case when the augmented row $i$ is of length $1$ in which case row $(r+1)$ of $U^\prime$ is the same as row $r$ of $U$.  Let $y$ be the entry that scans the top of the $j$th column.

We claim $y \geq U(r,j)$.  Suppose not.  Then $y<U(r,j)$.  When scanning column $j+1$ if the value $y$ was in hand at row $r$, we would have put $y$ in a new box with coordinates $(r, j+1)$.  Since this is not the case, $y$ was bumped from position $(s, j+1)$, $s>r$.  In this case $y=\hat{U}(s, j+1) >0=\hat{U}(r, j+1)$ with $y<\hat{U}(r,j)$.  This is a Triple Rule violation in $U$, thus $y\geq U(r,j)$.

If $j=1$ then since $y \geq U(r,j)$, $y$ would be inserted into a new row $i$ where $i>r$.  This is contrary to our assumption that the augmented row $i$ satisfies $r>i$. So we can assume $j>1$.

We must have $U(r,j)=U^\prime(r,j) \geq U(i, j-1)=U^\prime(i,j-1)$, or else $U$ would have a Triple Rule violation.  Since $U(i, j-1)=U^\prime(i, j-1)>U^\prime(i,j)$ we have $U^\prime(r,j) > U^\prime(i,j)$.

Consider now the portion of the insertion path in column $j$, say in rows $i_0 < i_1 <\ldots <i_t=i$, where $y=U^\prime(i_0, j)$.  Since $y\geq U(r,j)=U^\prime(r,j) > U^\prime(i,j)$ and since the entries in the insertion path are weakly decreasing, there is some index $\ell$, $0\leq \ell < t$, such that 
\begin{equation}{\label{star1}}
U^\prime(i_\ell, j)\geq U^\prime(r,j) > U^\prime(i_{\ell+1}, j).
\end{equation}
Since rows strictly decrease, 
\begin{equation}{\label{star2}}
U^\prime(i_\ell, j-1) > U^\prime(i_\ell, j) \geq U^\prime(r,j).
\end{equation}
Further, note that
\begin{equation}{\label{star3}}
U(i_p, j)=U^\prime(i_{p+1}, j) \text{ for all } 0\leq p < t.
\end{equation}
Now combining (\ref{star1}),(\ref{star2}), and (\ref{star3}) we get in $U$ the inequalities ${U(r,j)=U^\prime(r,j) > U^\prime(i_{\ell+1},j)= U(i_\ell, j)}$ but $U(r,j)=U^\prime(r,j)< U^\prime(i_\ell, j-1)=U(i_\ell, j-1)$, which is a Triple Rule violation in $U$.

Thus in all cases we obtain a contradiction.
\end{proof}

Consider the RCT obtained after $n$ successive insertions 
\begin{equation*}
U_n := (\cdots ((U \leftarrow b_1) \leftarrow b_2) \cdots) \leftarrow b_n
\end{equation*}
where the $b_i$ are arbitrary positive integers. Any row $i$ of $U_n$ will either consist entirely of boxes added during the successive insertions, or it will consist of some number of boxes from $U$ with some number of boxes added during the successive insertions.  In the former case row $i$ corresponds to some row $\hat{i}$ in each $U_j$ for $j > k\geq 1$, where $k$ is such that the insertion of $b_k$ adds a box in position $(\hat{i},1)$.  In the latter case row $i$ corresponds to some row $\hat{i}$ in each $U_j$ for all $0 \leq j \leq n$ where $U_0:=U$.  

As a direct consequence of Lemma \ref{lastrow} we have 

\begin{lemma}{\label{weakly longer}}
Consider $U_n$, the RCT obtained after $n$ successive insertions. Consider two rows $i$ and $i^\prime$ of $U_n$ such that $i < i^\prime$ and row $i$ is weakly longer than row $i^\prime$.  Suppose $b_{k_1}$ adds a box in position $(\hat{i},1)$ and $b_{k_2}$ adds a box in position $(\hat{i^\prime},1)$. Then $k_1<k_2$ and the corresponding row $\hat{i}$ is weakly longer than the corresponding row $\hat{i^\prime}$ in each $U_j$ for $j\geq k_2$.  
\end{lemma}

\begin{proof}
Suppose for a contradiction that at some intermediate step $U_j$ row $\hat{i}$ is strictly shorter than $\hat{i^\prime}$.  Since row $i$ is weakly longer than row $i^\prime$ in $U_n$, we must have that for some $\ell$, $j <\ell$, the new box produced in the insertion of $b_\ell$ into $U_{\ell-1}$ is at the end of the corresponding row $\hat{i}$ and rows $\hat{i}$ and $\hat{i^\prime}$ have the same length.  This contradicts Lemma \ref{lastrow}.

\end{proof}

Lemma \ref{lastrow} allows us to invert the insertion process for RCT's.  More specifically, given a RCT $U^\prime$ of shape $\alpha^\prime$ we can obtain a RCT $U$ of shape $\alpha$, where $\alpha^\prime=(\alpha_1, \ldots, \alpha_i+1, \ldots, \alpha_l)$ or $\alpha^\prime=(\alpha_1, \ldots, \alpha_{i-1}, 1, \alpha_i, \ldots, \alpha_l)$, in the following way.  We can un-insert the last entry, call it $y$, in row $i$ of $\alpha^\prime$, where row $i$ is the lowest row of length $j=\alpha_i+1$ or $j=1$. Do so by scanning up columns from bottom to top and un-bumping entries $\tilde{y}$ weakly greater than $y$ whenever $y$ is strictly greater than the entry to the right of $\tilde{y}$.  After scanning a column, we move one column to the right and continue scanning bottom to top.  In the end we will have un-inserted an entry $k$ such that $U^\prime=U \leftarrow k$.

\subsection{Main Bumping Property}

As above, let $U$ be a RCT with $k$ rows and longest row length $m$.  Consider $U \leftarrow b \leftarrow c$ with $b \leq c$.  Let $b_j^i$ be the entry ``in hand" which scans the entry in the $i$th row and $j$th column of $U$ during the insertion of $b$ into $U$, where $b_j^0$ is the element that begins scanning at the top of column $j$, so $b_{m+1}^0:=b$.  If the insertion of $b$ stops in position $(i_b,j_b)$ then $b_j^i:=0$ for all positions $(i,j) <_{col} (i_b,j_b)$ in $U$.  Similarly, let $c_j^i$ be the entry ``in hand" which we compare against the entry in the $i$th row and $j$th column of $U \leftarrow b$ during the insertion of $c$ into $U \leftarrow b$, where $c_j^0$ is the element that begins scanning the top of column $j$.  Note that $c$ will begin scanning in column $m+2$, since the insertion of $b$ may end in column $m+1$.  But when $b \leq c$ we have $c_{m+1}^0=c$ regardless of where the insertion of $b$ ends. If the insertion of $c$ into $U \leftarrow b$ stops in position $(i_c, j_c)$ we let $c_j^i:=0$ for all positions $(i,j)<_{col} (i_c, j_c)$ in $U \leftarrow b$. 

Now consider $U \leftarrow b \leftarrow a$ with $b>a$.  Define $b_j^i$ as above.  Similarly, we can define $a_j^i$ to be the entry which scans the entry in the $i$th row and $j$th column of $U \leftarrow b$ during the insertion of $a$ into $U \leftarrow b$.  Define $a_j^0$ to be the entry that begins scanning at the top of the $j$th column.  Define $a_{m+2}^0:=a$.  If the insertion of $a$ stops in position $(i_a, j_a)$ then let $a_j^i:=0$ for all positions $(i,j)<_{col} (i_a,j_a)$ in $U \leftarrow b$.  

\begin{lemma}{\label{scanning values}}
Let $U$ be a RCT with $k$ rows and longest row length $m$.  Let $a<b\leq c$ be positive integers. Suppose the insertion of $b$ into $U$ creates a new box in position $(i_b, j_b)$ in $U \leftarrow b$. The scanning values $b_j^i, c_j^i, a_j^i$ have the following relations.
\begin{enumerate}
\item Consider $U \leftarrow b \leftarrow c$.
	\begin{enumerate}
	\item If $U\leftarrow b$ has the same number of rows as $U$, then $b_j^i \leq c_j^i$ for all $(i,j)$ such that $0 \leq i \leq i_b	$ when $j=j_b$ and $0 \leq i \leq k$ when $j_b < j <m+1$.
	\item If $U\leftarrow b$ has one more row than $U$, that is $j_b=1$, then 
	\[
	\begin{array}{ll}
	b_j^i \leq c_j^i & \text{ for all } 0 \leq i \leq i_b \text{ and } 1 \leq j \leq m+1, \\
	b_j^{i} \leq c_j^{i+1} & \text{ for all } i_b \leq i \leq k+1 \text{ and } 2 \leq j \leq m+1.
	\end{array}
	\]
	\end{enumerate}
\item Consider $U \leftarrow b \leftarrow a$.
	\begin{enumerate}
	\item If $U\leftarrow b$ has the same number of rows as $U$, then $b_j^i > a_{j+1}^i$ for all $(i,j)$ such that $0 \leq i \leq i_b$ when $j=j_b$ and $0 \leq i \leq k$ when $j_b < j \leq m$.
	\item If $U\leftarrow b$ has one more row than $U$, that is $j_b=1$, then 
	\[
	\begin{array}{ll}
	b_j^i > a_{j+1}^i & \text{ for all } 0 \leq i \leq i_b \text{ and } 1 \leq j \leq m, \\
	b_j^i > a_{j+1}^{i+1} & \text{ for all } i_b \leq i \leq k+1 \text{ and } 2 \leq j \leq m+1.
	\end{array}
	\]
	\end{enumerate}
\end{enumerate}
\end{lemma}

\begin{remark}
Informally, Lemma \ref{scanning values} states that when doing consecutive insertions $U \leftarrow b \leftarrow c$ or $U \leftarrow b \leftarrow a$, the scanning values created by $b$ are weakly less than the scanning values created by $c$, and the scanning values of $b$ are strictly greater than the scanning values created by $a$.  Note that $b_j^i > a_{j+1}^i$ implies $b_j^i > a_j^i$ for $(i,j)$ satisfying the conditions of Lemma \ref{scanning values} part (2).
\end{remark}

\begin{proof}
\textbf{Proof of (1a):} Let $j=m+1$ and $i=0$.  Then clearly $b_{m+1}^0=b \leq c_{m+1}^0 =c$.  Now fix the column index $j>j_b$. Suppose by induction that $b_j^p \leq c_j^p$ for all $p\leq i$.  To show $b_j^{i+1} \leq c_j^{i+1}$ consider the following cases.

\emph{Case 1:} Suppose $b_j^i$ bumps the entry $b_j^{i+1}$ in position $(i, j)$ of $U$, and $c_j^i$ does not bump in position $(i,j)$ of $U^\prime$.  In this case, $b_j^{i+1} \leq b_j^i \leq c_j^i = c_j^{i+1}$.

\emph{Case 2:}  Suppose $b_j^i$ bumps the entry $b_j^{i+1}$ in position $(i, j)$ of $U$, and $c_j^i$ bumps the entry $c_j^{i+1}=b_j^i$ in position $(i,j)$ of $U^\prime$.  Then $b_j^{i+1} \leq b_j^i =c_j^{i+1}$.

\emph{Case 3:} Suppose neither $b_j^i$ nor $c_j^i$ bump in position $(i,j)$ of their respective RCT.  Then $b_j^{i+1}=b_j^i \leq c_j^i = c_j^{i+1}$.

\emph{Case 4:} Suppose $b_j^i$ does not bump in position $(i,j)$ of $U$, but $c_j^i$ bumps $c_j^{i+1}$ in position $(i,j)$ of $U^\prime$.  Consider the following diagram which depicts row $i$ and columns $j-1$ and $j$ in each of $U$, $U \leftarrow b$, and $U \leftarrow b \leftarrow c$.

\[
\begin{array}{ccc}
U & U \leftarrow b & U \leftarrow b \leftarrow c \\
\Tableau{ d & c_j^{i+1}} & \Tableau{ \tilde{d} & c_j^{i+1} } & \Tableau{ \tilde{d} & c_j^i }
\end{array}
\]

If $d$ is bumped by $\tilde{d}$ during the insertion of $b$, then $d \leq \tilde{d} \leq b_{j-1}^0 \leq b_j^i \leq c_j^i$ which contradicts row strictness of $U \leftarrow b \leftarrow c$.  So assume $d$ does not get bumped by $\tilde{d}$, that is $d=\tilde{d}$.  We get $d= \tilde{d} > c_j^i \geq b_j^i$ and since $b_j^i$ does not bump we must have $b_j^i < c_j^{i+1}$.  But then $b_j^i =b_j^{i+1} < c_j^{i+1}$.

The argument above shows that for fixed $j$, $b_j^i \leq c_j^i$ for all $0 \leq i \leq k$.  But this immediately implies $b_{j-1}^0 \leq c_{j-1}^0$ and thus we have $b_j^i \leq c_j^i$ for all $(i,j)$ indicated in the lemma.

\textbf{Proof of (1b):} Notice that row $i+1$ in $U^\prime$ will correspond to row $i$ in $U$ for all $i_b \leq i \leq k+1$.  Since row $i_b$ has only one box in it, then $c_j^{i_b}=c_j^{i_b+1}$ for $3\leq j \leq m+1$.  So assume $j$ is fixed such that $3 \leq j \leq m+1$.  The proof for part (1a) establishes $b_j^i \leq c_j^i$ for $0 \leq i \leq i_b$, which immediately implies $b_j^{i_b} \leq c_j^{i_b+1}$.  

Now suppose by induction that $b_j^p \leq c_j^{p+1}$ for all $p$ such that $i_b\leq p \leq i$ for some $i$.  We establish $b_j^{i+1} \leq c_j^{i+2}$ by considering the following cases.

\emph{Case 1:} Suppose $b_j^{i}$ bumps the entry $b_j^{i+1}$ in position $(i, j)$ of $U$, and $c_j^{i+1}$ does not bump in position $(i+1,j)$ of $U^\prime$.  Then $b_j^{i+1} \leq b_j^{i} \leq c_j^{i+1} =c_j^{i+2}$.

\emph{Case 2:} Suppose $b_j^{i}$ bumps the entry $b_j^{i+1}$ in position $(i, j)$ of $U$, and $c_j^{i+1}$ bumps the entry $c_j^{i+2}=b_j^{i}$ in position $(i+1,j)$ of $U^\prime$.  Then $b_j^{i+1} \leq b_j^i = c_j^{i+2}$.

\emph{Case 3:} Suppose neither $b_j^{i}$ does not bump in position $(i,j)$ of $U$ and $c_j^{i+1}$ does not bump in position $(i+1,j)$ of $U^\prime$.  Then $b_j^{i+1}=b_j^{i} \leq c_j^{i+1}=c_j^{i+2}$.

\emph{Case 4:} Suppose $b_j^{i}$ does not bump in position $(i,j)$ of $U$, but $c_j^{i+1}$ bumps $c_j^{i+2}$ in position $(i+1,j)$ of $U^\prime$. Consider the following diagram which depicts columns $j-1$ and $j$ and the labelled rows of $U$, $U \leftarrow b$, and $U\leftarrow b \leftarrow c$.

\[
\begin{array}{lccc}
 & U & U \leftarrow b & U \leftarrow b \leftarrow c \\
i& \Tableau{ d & c_j^{i+2}} & \vdots & \vdots  \\
i+1& \vdots & \Tableau{ \tilde{d} & c_j^{i+2} } & \Tableau{ \tilde{d} & c_j^{i+1} }
\end{array}
\]

If $d$ is bumped by $\tilde{d}$ during the insertion of $b$, then $d \leq \tilde{d} \leq b_{j-1}^0 \leq b_j^{i} \leq c_j^{i+1}$ which contradicts row strictness of $U \leftarrow b \leftarrow c$.  So assume $d$ does not get bumped by $\tilde{d}$, that is $d=\tilde{d}$.  We get $d= \tilde{d} > c_j^{i+1} \geq b_j^{i}$ and since $b_j^{i}$ does not bump we must have $b_j^{i} < c_j^{i+2}$.  But then $b_j^{i+1} =b_j^{i} < c_j^{i+2}$.

When $j=2$, the above argument shows $b_2^i \leq c_2^i$ for all $0 \leq i \leq i_b$, which implies the insertion of $c$ cannot add a new box with entry $c_2^{i_b}$ in position $(i_b, 2)$ of $U^\prime$, otherwise $U^\prime(i_b,1) \leq b_2^{i_b} \leq c_2^{i_b}$.  So $c_2^{i_b}=c_2^{i_b+1}$ and the above argument shows $b_2^i \leq c_2^{i+1}$ for all $i_b \leq i \leq k+1$.

The case of $j=2$ implies $b_1^0 \leq c_1^0$.  The definition of insertion implies $b_1^i=b_1^0$ for all $0 \leq i \leq i_b$, and $c_1^i=c_1^0$ for all $0 \leq i \leq i_b$.  Thus, the relations in part (1b) of the lemma follow. 

\textbf{Proof of (2a)} We have $b_{m+1}^0 > a_{m+2}^0$ by assumption.  Now fix a column $j>j_b$ and assume by induction that $b_j^p > a_{j+1}^p$ for all $p\leq i$ for some $i$.  To show $b_j^{i+1} > a_{j+1}^{i+1}$ we consider the following cases.

\emph{Case 1:} Suppose $b_j^i$ bumps the entry $b_j^{i+1}$ in position $(i,j)$ of $U$, and $a_{j+1}^i$ bumps the entry $a_{j+1}^{i+1}$ in position $(i, j+1)$ of $U \leftarrow b$.  Then $U(i,j)=b_j^{i+1} > U(i,j+1)=a_j^{i+1}$ by row-strictness of $U$.

\emph{Case 2:} Suppose $b_j^i$ does not bump in position $(i,j)$ of $U$, and $a_{j+1}^i$ bumps the entry $a_{j+1}^{i+1}$ in position $(i, j+1)$ of $U \leftarrow b$.  Then $b_j^{i+1}=b_j^i > a_{j+1}^i \geq a_{j+1}^{i+1}$.

\emph{Case 3:} Suppose $b_j^i$ does not bump in position $(i,j)$ of $U$, and $a_{j+1}^i$ does not bump in position $(i, j+1)$ of $U\leftarrow b$.  Then $b_j^{i+1}=b_j^i > a_{j+1}^i = a_{j+1}^{i+1}$.

\emph{Case 4:} Suppose $b_j^i$ bumps the entry $b_j^{i+1}$ in position $(i,j)$ of $U$, and $a_{j+1}^i$ does not bump in position $(i, j+1)$ of $U\leftarrow b$. Let $U^\prime=U \leftarrow b$.  Then 
\begin{displaymath}
U^\prime(i,j)=b_j^i \geq b_j^{i+1}=U(i,j) > U^\prime(i,j+1)=U(i,j+1).
\end{displaymath}  
Because $U^\prime(i,j)=b_j^i>a_{j+1}^i$ and $a_{j+1}^i$ does not bump, we must have $a_{j+1}^i < U^\prime(i,j+1)$.  This implies $b_j^{i+1}>a_{j+1}^i=a_{j+1}^{i+1}$.

The argument above shows that for fixed $j$,  $b_j^i > a_{j+1}^i$ for all $0 \leq i \leq k$.  This implies $b_{j-1}^0 > a_j^0$, which then implies the relations in part (2a) of the lemma.

\textbf{Proof of (2b):} Notice that row $i+1$ in $U^\prime$ will correspond to row $i$ in $U$ for all $i_b \leq i \leq k+1$.  Since row $i_b$ has only one box in it, then $a_{j+1}^{i_b}=a_{j+1}^{i_b+1}$ for $2\leq j \leq m+1$.  So assume $j$ is fixed such that $2 \leq j \leq m+1$.  The proof for part (2a) establishes $b_j^i > a_{j+1}^i$ for $0 \leq i \leq i_b$, which immediately implies $b_j^{i_b} > a_{j+1}^{i_b+1}$.  

Now suppose by induction that $b_j^p > a_{j+1}^{p+1}$ for all $p$ such that $i_b\leq p \leq i$ for some $i$.  We establish $b_j^{i+1} > a_{j+1}^{i+2}$ by considering the following cases.

\emph{Case 1:} Suppose $b_j^i$ bumps the entry $b_j^{i+1}$ in position $(i,j)$ of $U$, and $a_{j+1}^{i+1}$ bumps the entry $a_{j+1}^{i+2}$ in position $(i+1, j+1)$ of $U \leftarrow b$.  Then $U(i,j)=b_j^{i+1}> U(i,j+1)=U^\prime(i+1, j+1)=a_j^{i+2}$. 

\emph{Case 2:}  Suppose $b_j^i$ does not bump in position $(i,j)$ of $U$, and $a_{j+1}^{i+1}$ bumps the entry $a_{j+1}^{i+2}$ in position $(i+1, j+1)$ of $U \leftarrow b$. Then $b_j^{i+1}=b_j^i>a_{j+1}^{i+1} \geq a_{j+1}^{i+2}$.

\emph{Case 3:} Suppose $b_j^i$ does not bump in position $(i,j)$ of $U$, and $a_{j+1}^{i+1}$ does not bump in position $(i+1, j+1)$ of $U\leftarrow b$. Then $b_j^{i+1}=b_j^i>a_{j+1}^{i+1} = a_{j+1}^{i+2}$.

\emph{Case 4:} Suppose $b_j^i$ bumps the entry $b_j^{i+1}$ in position $(i,j)$ of $U$, and $a_{j+1}^{i+1}$ does not bump in position $(i+1, j+1)$ of $U\leftarrow b$. Let $U^\prime=U \leftarrow b$. Then
\begin{displaymath}
U^\prime(i+1, j)=b_j^i \geq b_j^{i+1} = U(i,j) > U(i,j+1)=U^\prime(i+1, j+1).
\end{displaymath}
Because $U^\prime(i+1,j)=b_j^i>a_{j+1}^{i+1}$ and $a_{j+1}^{i+1}$ does not bump, we must have $a_{j+1}^{i+1}<U^\prime(i+1,j+1)$. This implies $b_j^{i+1}> a_{j+1}^{i+1}=a_{j+1}^{i+2}$.

In the case $j=1$, the definition of RCT insertion forces each scanning value $b_1^i=b_1^0$ for all rows $0 \leq i \leq i_b$.  Since $b_1^0 > a_2^0$ by the argument above, and since the entries bumped by $a_2^0$ in the second column get weakly smaller we have $b_1^i>a_2^i$ for all $0 \leq i \leq i_b$ as needed.
\end{proof}

We can apply Lemma \ref{scanning values} to prove the following proposition, which describes where new boxes are added after consecutive insertions. 

\begin{proposition}{\label{bumping}}
Let $U$ be a RCT with $k$ rows, longest row length $m$.  Let $a$, $b$, and $c$ be positive integers with $a<b\leq c$.  Consider successive insertions $U_1 := (U \leftarrow b)\leftarrow c$ and $U_2:=(U \leftarrow b) \leftarrow a$.  Let $B_a=(i_a, j_a), B_b=(i_b, j_b),$ and $B_c=(i_c, j_c)$ be the new boxes created after inserting $a,b,$ and $c$, respectively, into the appropriate RCT.  Let $i_1$ be a row in $U_1$ which contains a box $(i_1, j_1)$ from $I(b)$ and a box $(i_1, j_1^\prime)$ from $I(c)$.  Similarly, let $i_2$ be a row in $U_2$ which contains a box $(i_2, j_2)$ from $I(b)$ and a box $(i_2, j_2^\prime)$ from $I(a)$.   Then 
\begin{enumerate}
\item In $U_1$, $j_c \leq j_b$.  In $U_2$, $j_a > j_b$.  
\item In $U_1$, $j_1^\prime \leq j_1$.  In $U_2$, $j_2^\prime > j_2$.
\end{enumerate}
\end{proposition}

\begin{remark} Informally, part (1) of Proposition \ref{bumping} states that if $a < b \leq c$, then in $U \leftarrow b \leftarrow c$ the new box created by $c$ is weakly left of the new box created by $b$, and in $U \leftarrow b \leftarrow a$ the new box created by $a$ is strictly right of the new box created by $b$. Part (2) of Proposition \ref{bumping} states that the insertion path of $c$ is weakly left of the insertion path of $b$, and the insertion path of $a$ is strictly right of the insertion path of $b$.
\end{remark}

\begin{proof}
\textbf{Proof of (1):} Lemma \ref{scanning values} part (1) shows that during the insertion of $c$ into $U \leftarrow b$, the scanning values $c_{j_b}^{i_b}$ is weakly greater than the entry occupying the box $B_b$, which forces the new box $B_c$ to be weakly left of $B_b$, that is $j_c\leq j_b$.  Lemma \ref{scanning values} part (2) show that during the insertion of $a$ into $U \leftarrow b$, the new box $B_a$ must occupy position $(i_b, j_b+1)$ if the insertion process reaches this position, implying that the new box $B_a$ is always strictly right of the box $B_b$, that is $j_a> j_b$.

\textbf{Proof of (2):} Suppose for a contradiction that there is a row $i_1$ of $U_1$ which contains a box $(i_1, j_1)$ from $I(b)$ and a box $(i_1, j_1^\prime)$ from $I(c)$, and that $j_1^\prime >j_1$.  Then 
\begin{displaymath}
U_1(i_1,j_1) \leq b_{j_1}^0 \leq c_{j_1}^0 \leq U_1(i_1, j_1^\prime)
\end{displaymath}  
which contradicts row-strictness in $U_1$.

Again, suppose for a contradiction that there is a row $i_2$ in $U_2$ which contains a box $(i_2, j_2)$ from $I(b)$ and a box $(i_2, j_2^\prime)$ from $I(a)$ and $j_2^\prime \leq j_2$.  If the boxes coincide, that is $j_2=j_2^\prime$, then $a_{j_2}^{i_2}$ bumped the entry $b_{j_2}^{i_2}$ in position $(i_2, j_2)$ of $U \leftarrow b$, and Lemma \ref{scanning values} shows $b_{j_2}^{i_2} > a_{j_2+1}^{i_2} \geq a_{j_2}^{i_2}$, which contradicts the definition of RCT insertion.  If $j_2^\prime < j_2$ then
\begin{displaymath}
U_2(i_2, j_2^\prime) \leq a_{j_2^\prime}^0< b_{j_2^\prime}^0 \leq U_2(i_2, j_2)
\end{displaymath}
where $a_{j_2^\prime}^0 < b_{j_2^\prime}^0$ is established by using Lemma \ref{scanning values}. But this contradicts row-strictness of $U_2$.
\end{proof}

The following lemma follows from Proposition \ref{bumping}.

\begin{lemma}{\label{readingorder}}
Consider the RCT obtained after $n$ successive insertions 
\begin{displaymath}
U_n := (\cdots ((U \leftarrow b_1) \leftarrow b_2) \cdots) \leftarrow b_n
\end{displaymath}
with $b_1\leq b_2 \leq \cdots \leq b_n$ positive integers. Let $B_1, B_2, \ldots, B_n$ be the corresponding new boxes.  Then in $U_n$, 
\begin{displaymath}
B_n <_{col} B_{n-1} <_{col} \cdots <_{col} B_1.
\end{displaymath}
\end{lemma}

\begin{proof}
Proposition \ref{bumping} implies the new boxes are added weakly right to left.  Let $i_1 <i_2$ and consider two boxes $B_{i_1}$ and $B_{i_2}$ in the same column.  Note that $B_{i_1}$ and $B_{i_2}$ cannot coincide.  Suppose for a contradiction that $B_{i_1}$ is (strictly) below $B_{i_2}$.  Suppose the row containing box $B_{i_1}$ has length $j$.  Once $B_{i_1}$ is added, the new boxes $B_k$  for $i_1<k<i_2$ cannot change the length of the row containing $B_{i_1}$.    Thus, when $B_{i_2}$ is added to the end of a row of length $j-1$ strictly above the row containing $B_{i_1}$, we contradict Lemma \ref{lastrow}.
\end{proof}

\subsection{Elementary Transformations}

Knuth's contribution to the RSK algorithm included describing Schensted insertion in terms of two \emph{elementary transformations $\kone$ and $\ktwo$} which act on words $w$.  Let $a,b,$ and $c$ be positive integers.  Then

\begin{displaymath}
\begin{array}{lll}
\mathcal{K}_1: & bca \to bac & \text{ if } a<b \leq c \\
\mathcal{K}_2: & acb \to cab & \text{ if } a \leq b < c 
\end{array}.
\end{displaymath}

The relations $\kone, \ktwo,$ and their inverses $\kone^{-1}, \ktwo^{-1}$, act on words $w$ by transforming triples of consecutive letters.  Denote by $\stackrel{1}{\cong}$ the equivalence relation defined by using $\kone$ and $\kone^{-1}$.  That is, $w \stackrel{1}{\cong} w^\prime$ if and only if $w$ can be transformed into $w^\prime$ using a finite sequence of transformations $\kone$ or $\kone^{-1}$.

\begin{lemma}{\label{knuth}}
Let $U$ be a RCT and let $w$ and $w^\prime$ be two words such that $w \stackrel{1}{\cong} w^\prime$.  Then
\begin{displaymath}
U \leftarrow w = U \leftarrow w^\prime.
\end{displaymath}
\end{lemma}

\begin{proof}
It suffices to show
\begin{displaymath}
U \leftarrow b \leftarrow c \leftarrow a = U \leftarrow b \leftarrow a \leftarrow c
\end{displaymath}
for positive integers $a < b \leq c$.

Consider the insertion path $I(c)$ in $U \leftarrow b \leftarrow c$.  By Proposition \ref{bumping} we know that in $U \leftarrow b \leftarrow a$ the insertion of $a$ cannot end in a new box in the first column.  Thus, for each box $(i,j)$ in $I(c)$ we may consider the corresponding box $(i,j)$ in $U \leftarrow b \leftarrow a$.

We will inductively show that the insertion path of $c$ when inserting into $U \leftarrow b \leftarrow a$ is the exact same set of boxes $I(c)$ and bumps exactly the same set of entries in these boxes. Let $(i,j)$ be the largest box with respect to $<_{col}$ in $U \leftarrow b \leftarrow c$, that is, the box $(i,j)$ is the first box used in the insertion of $c$ into $U \leftarrow b$.  The base case is established in three steps.

\emph{Step 1:} The entry $c$ cannot bump in a box $(p,q)$ such that $(i,j) <_{col} (p,q)$ in $U \leftarrow b \leftarrow a$.  Suppose for a contradiction that $c$ does bump in box $(p,q)$. Then, by definition of insertion, we get $q\neq 1$. This implies the box $(p,q-1)$ is in the path of $a$, which places the path of $a$ strictly to the left of the path of $c$ in $U \leftarrow b \leftarrow a \leftarrow c$, which contradictions Proposition \ref{bumping}.

\emph{Step 2:} The box $(i,j)$ in $U \leftarrow b \leftarrow a \leftarrow c$ must be in the path of $c$.  Suppose for a contradiction that the box $(i,j)$ is not in the path of $c$. Under these assumptions $j \neq 1$.  Then this implies that $(i,j)$ is also in the path of $a$.  Let $d$ be the entry in position $(i,j-1)$ of $U \leftarrow b$, $U \leftarrow b \leftarrow c$, and $U \leftarrow b \leftarrow a$. Since $(i,j)$ is in the path of $c$ in $U \leftarrow b \leftarrow c$, then $d > c$.  Similarly, $d > a_j^i$.  Since $(i,j)$ is not in the path of $c$ in $U \leftarrow b \leftarrow a  \leftarrow c$ by assumption, and since $d$ is still the entry in position $(i,j-1)$ when $c$ scans it, we must have $c=c_j^i < a_j^i$, which contradicts Lemma \ref{scanning values}.

\emph{Step 3:} This step only needs to be checked when $j \geq 2$.  We claim the entry $c$ bumps the same value in box $(i,j)$ in both $U \leftarrow b  \leftarrow c$ and $U \leftarrow b  \leftarrow a  \leftarrow c$.  Suppose for a contradiction that $c$ bumps a different value in $(i,j)$ of $U\leftarrow b  \leftarrow a \leftarrow c$. Then this implies the box $(i,j)$ is in the path of $a$ and (by assumption) in the path of $c$. Consider the following diagram, which depicts boxes $(i,j-1)$ and $(i,j)$.

\[
\begin{array}{ccccc}
U & U \leftarrow b & U \leftarrow b \leftarrow c & U \leftarrow b  \leftarrow a & U \leftarrow b \leftarrow a \leftarrow c \\
\Tableau{ d & y} & \Tableau{ \tilde{d} & \tilde{y} } & \Tableau{ \tilde{d} & c } & \Tableau{\tilde{d} & a_j^i } & \Tableau{\tilde{d} & c}
\end{array}
\]

Using Lemma \ref{scanning values} we get $b_j^i > a_{j+1}^i$ (or $b_j^i > a_{j+1}^{i+1}$ in the appropriate rows if the insertion of $b$ into $U$ created a new row). Thus  $b_j^i > a_j^i \geq \tilde{y} \geq y$.  We can also establish $b_j^i \leq c_j^i=c$ (or $b_j^i \leq c_j^{i+1} \leq c_j^i=c$).  In the case $y=\tilde{y}$ then either $\tilde{d}=b_{j-1}^i \leq b_j^i \leq c < \tilde{d}$ which is a contradiction, or $\tilde{d}=d>c \geq b_j^i$.  In the case $\tilde{y}=b_j^i$ then $d=\tilde{d}> c \geq b_j^i$. In all cases we have $d>b_j^i$ and $b_j^i \geq y$, which implies $b_j^i$ must bump in position $(i,j)$.  This immediately implies (by way of Proposition \ref{bumping}) that $a$ cannot have the box $(i,j)$ in its bumping path.

This completes the base case.  To finish the proof induct on the path $I(c)$. Suppose by induction that the path of $c$ in $U \leftarrow b \leftarrow c$ is identical to the path of $c$ in $U \leftarrow b \leftarrow a \leftarrow c$, and the bumped entries are the same in both paths, up to some box $(i,j)$ where the path, or the value bumped, is different.  Under the inductive hypothesis the scanning values $c_j^i$ obtained when inserting $c$ into $U \leftarrow b$ are equal to the scanning values, also denoted $c_j^i$, obtained when inserting $c$ into $U \leftarrow b \leftarrow a$. We show $(i,j)$ in $U \leftarrow b \leftarrow a \leftarrow c$ is in the path of $c$ if and only if $(i,j)$ in $U\leftarrow b \leftarrow c$ is in the path of $c$, and $c_j^i$ bumps the same valued entry.  We do this in three steps which are identical to the three steps above.

\emph{Step 1:} If $(i,j)$ is in the path of $c$ in $U \leftarrow b \leftarrow a \leftarrow c$, then $(i,j)$ is in the path of $c$ in $U \leftarrow b \leftarrow c$.  This is clearly true if $j=1$.  When $j \geq 2$ and if this were not the case, that is $(i,j)$ is not in the path of $c$ in $U \leftarrow b \leftarrow c$, then the entry $a$ must have $(i,j-1)$ in its insertion path in $U \leftarrow b \leftarrow a$, which places the path of $a$ strictly to the left of the path of $c$ in $U \leftarrow b \leftarrow a \leftarrow c$ which is a contradiction.

\emph{Step 2:} If $(i,j)$ is in the path of $c$ in $U \leftarrow b \leftarrow c$ then $(i,j)$ is in the path of $c$ in $U \leftarrow b \leftarrow a \leftarrow c$. Again, this is clearly true if $j=1$.  When $j \geq 2$ and if this were not the case, that is $(i,j)$ is not in the path of $c$ in $U \leftarrow b \leftarrow a \leftarrow c$, then the entry $a$ must have bumped in position $(i,j)$. Let $d$ be the entry in box $(i,j-1)$ of $U \leftarrow b$, $U \leftarrow b \leftarrow c$, and $U \leftarrow b \leftarrow a$. The fact that $(i,j)$ is in the path of $c$ in $U \leftarrow b \leftarrow c$ implies $d > c_j^i$.  Similarly, $d > a_j^i$.  Under our assumptions $(i,j)$ is not in the path of $c$ in $U \leftarrow b \leftarrow a \leftarrow c$, which implies $c_j^i < a_j^i$ which contradicts Lemma \ref{scanning values}.

\emph{Step 3:} This step only needs to be checked if $j \geq 2$.  We claim the same value is bumped in box $(i,j)$ of $U\leftarrow b \leftarrow c$ and $U \leftarrow b \leftarrow a \leftarrow c$.  If this were not the case then we must have that both $a$ and $c$ have the box $(i,j)$ in their respective insertion paths in $U \leftarrow b \leftarrow a \leftarrow c$.  Consider the following diagram which depicts boxes $(i,j-1)$ and $(i,j)$.

\[
\begin{array}{ccccc}
U & U \leftarrow b & U \leftarrow b \leftarrow c & U \leftarrow b  \leftarrow a & U \leftarrow b \leftarrow a \leftarrow c \\
\Tableau{ d & y} & \Tableau{ \tilde{d} & \tilde{y} } & \Tableau{ \tilde{d} & c_j^i } & \Tableau{\tilde{d} & a_j^i } & \Tableau{\tilde{d} & c_j^i}
\end{array}
\]

Using Lemma \ref{scanning values} we get $b_j^i > a_{j+1}^i$ (or $b_j^i > a_{j+1}^{i+1}$ in the appropriate rows if the insertion of $b$ into $U$ created a new row). Thus  $b_j^i > a_j^i \geq \tilde{y} \geq y$.  We can also establish $b_j^i \leq c_j^i$ (or $b_j^i \leq c_j^{i+1} \leq c_j^i$).  In the case $y=\tilde{y}$ then either $\tilde{d}=b_{j-1}^i \leq b_j^i \leq c_j^i < \tilde{d}$ which is a contradiction, or $\tilde{d}=d>c_j^i \geq b_j^i$.  In the case $\tilde{y}=b_j^i$ then $d=\tilde{d}> c_j^i \geq b_j^i$. In all cases we have $d>b_j^i$ and $b_j^i \geq y$, which implies $b_j^i$ must bump in position $(i,j)$.  This immediately implies (by way of Proposition \ref{bumping}) that $a$ cannot have the box $(i,j)$ in its bumping path.

Thus the path of $c$ in $U \leftarrow b \leftarrow c$ is exactly the same set of boxes as the path of $c$ in $U \leftarrow b \leftarrow a \leftarrow c$, and in both cases the same valued entries are bumped.

Now consider the insertion path $I(a)$ of $a$ in $U \leftarrow b \leftarrow a$.  By Proposition \ref{bumping} the insertion of $c$ into $U \leftarrow b$ may create a new box in the first column.  Despite this we may still consider the boxes in $U\leftarrow b \leftarrow c$ that correspond to the boxes in $I(a)$ since any particular box $(i,j)$ in $I(a)$ corresponds to the box $(i,j)$ in $U \leftarrow b \leftarrow c$ if row $i$ is above the new row created by $c$, or $(i,j)$ in $I(a)$ corresponds to $(i+1, j)$ in $U \leftarrow b \leftarrow c$ if row $i$ is weakly below the new row created by $c$.  With this in mind we will denote by $\widehat{(i,j)}$ the box in $U\leftarrow b \leftarrow c$ that corresponds to the box $(i,j)$ in $I(a)$.

Note that in both $U \leftarrow b \leftarrow a$ and $U \leftarrow b \leftarrow c \leftarrow a$ the path of $a$ cannot contain a box in the first column of the respective RCT.

We will inductively show that the path of $a$ in both $U \leftarrow b \leftarrow a$ and $U \leftarrow b \leftarrow c \leftarrow a$ consists of the the same (corresponding) boxes and the entries bumped in each path are equal entry by entry. Let $(i,j)$ be the largest box in $I(a)$ with respect to $<_{col}$. The base case can be established in three steps.

\emph{Step 1:} The entry $a$ cannot bump before the box $\widehat{(i,j)}$.  Suppose for a contradiction that $a$ bumped in some box $(p,q)$ with $\widehat{(i,j)}<_{col}(p,q)$. This implies that the box $(p, q-1)$ is in the path of $c$ in $U \leftarrow b \leftarrow c \leftarrow a$. 

Assume $q-1 \geq 2$.  Consider the diagram below, which depicts boxes $(p, q-2), (p, q-1)$, and $(p,q)$.

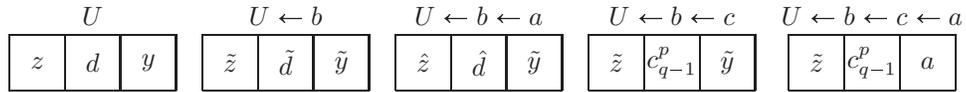
\[
\begin{array}{ccccc}
U & U \leftarrow b & U \leftarrow b \leftarrow a & U \leftarrow b  \leftarrow c & U \leftarrow b \leftarrow c \leftarrow a \\
\Tableau{ z & d & y} & \Tableau{ \tilde{z} & \tilde{d} & \tilde{y} } & \Tableau{ \hat{z} & \hat{d} & \tilde{y} } & \Tableau{\tilde{z} & c_{q-1}^p & \tilde{y} } & \Tableau{\tilde{z} & c_{q-1}^p & a}
\end{array}
\]

where either $\hat{z}$ or $\hat{d}$ could equal $a$ (but not both). We then get the inequalities $a \geq \tilde{y}$ which forces $a \geq \tilde{d}$. By Lemma \ref{scanning values} we know $b_{q-1}^p > a$, which implies $b_{q-1}^p > \tilde{d}\geq d$.  Thus the box $(p, q-1)$ is not in the path of $b$ and $\tilde{d}=d$.  On the other hand we see $\tilde{z} > c_{q-1}^p \geq b_{q-1}^p$, and since the path of $c$ cannot be strictly right of the path of $b$ we also see the box $(p, q-2)$ is not in the path of $b$ and thus $z=\tilde{z}$.  In the end we get the relations $z> b_{q-1}^p$ and $b_{q-1}^p> d$ which implies the box $(p, q-1)$ is in the path of $b$ and is a contradiction to the previously established condition on the box $(p, q-1)$.

Now we can assume $q-1=1$.  In this case the box $(p, q-1)=(p,1)$ is still in the path of $c$ and the position $(p,q)=(p, 2)$ is empty during the insertion of $a$.  With our assumptions that $(p,2)$ is in the path of $a$ in $U \leftarrow b \leftarrow c \leftarrow a$ and $\widehat{(i,j)} <_{col} (p,q)$ this forces $j=2$ and the insertion of $b$ must have created a new box in the first column, say in position $(r, 1)$ with $r<p$.  This means position $(r, 2)$ is empty during the insertion of $a$ and by Lemma \ref{scanning values}, $b_q^r > a$ and thus $a$ must insert in position $(r,2)$. Which means $a$ cannot have $(p,q)$ in its path.

\emph{Step 2:} The entry $a$ must bump in box $\widehat{(i,j)}$. Suppose for a contradiction that $a$ does not bump in box $\widehat{(i,j)}$ during the insertion of $a$ into $U \leftarrow b \leftarrow c$.  If $a$ does not bump in box $\widehat{(i,j)}$ then we must have $\widehat{(i,j)}$ in the path of $c$. As indicated above, $j\neq 1$.

Consider the following diagram which depicts boxes $(i,j-1)$ and $(i,j)$ and the corresponding boxes $\widehat{(i,j-1)}$ and $\widehat{(i,j)}$.

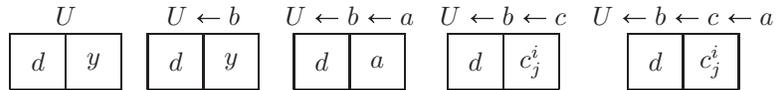
\[
\begin{array}{ccccc}
U&U \leftarrow b & U \leftarrow b \leftarrow a & U \leftarrow b  \leftarrow c & U \leftarrow b \leftarrow c \leftarrow a \\
 \Tableau{ d&y}&\Tableau{ d & y } & \Tableau{ d & a } & \Tableau{ d & c_j^i } & \Tableau{d & c_j^i}
\end{array}
\]

where neither box $(i,j)$ nor $(i,j-1)$ can be in the path of $b$ since the box $(i,j)$ is in the path of $a$ and $\widehat{(i,j)}$ is in the path of $c$. From our assumptions we get the inequalities $d > a \geq y$ and $d > c_j^i>a$.  Now consider the scanning values obtained during the insertion of $b$.  Lemma \ref{scanning values} implies $b_j^i > a_{j+1}^i=a \geq y$ and $d > c_j^i \geq b_j^i$.  These inequalities force the box $(i,j)$ to be in the path of $b$, which contradicts properties previously established. This implies $a$ must have $\widehat{(i,j)}$ in its insertion path in $U \leftarrow b \leftarrow c \leftarrow a$.

\emph{Step 3:} The entry $a$ bumps the same entry in box $(i,j)$ in $U \leftarrow b \leftarrow a$ as $a$ bumps in box $\widehat{(i,j)}$ in $U \leftarrow b \leftarrow c \leftarrow a$.  Suppose for a contradiction that $a$ bumps a different entry in box $\widehat{(i,j)}$ during the insertion of $a$ into $U \leftarrow b \leftarrow c$.  This implies $\widehat{(i,j)}$ is in the path of $c$ (and by assumption in the path of $a$). But this contradicts Proposition \ref{bumping}, as the path of $a$ must be strictly rightly right of the path of $c$.

Now induct on the boxes in $I(a)$. Suppose by induction that the path of $a$ in $U \leftarrow b \leftarrow a$ is identical to the path of $a$ in $U \leftarrow b \leftarrow c \leftarrow a$, and the bumped entries are the same in both paths, up to some box $(i,j)$ where the path, or the value bumped, is different.  Under the inductive hypothesis the scanning values $a_j^i$ obtained when inserting $a$ into $U \leftarrow b$ are equal to the scanning values, also denoted $a_j^i$, obtained when inserting $a$ into $U \leftarrow b \leftarrow c$. We show $(i,j)$ in $U \leftarrow b \leftarrow a \leftarrow c$ is in the path of $a$ if and only if $\widehat{(i,j)}$ in $U\leftarrow b \leftarrow c$ is in the path of $a$, and $a_j^i$ bumps the same valued entry.  We do this in three steps which are identical to the three steps above.

\emph{Step 1:} If $\widehat{(i,j)}$ is in the path of $a$ in $U \leftarrow b \leftarrow c \leftarrow a$ then $(i,j)$ must be in the path of $a$ in $U \leftarrow b \leftarrow a$.  If this were not the case then the box $\widehat{(i,j-1)}$ is in the path of $c$.

Assume $j-1 \geq 2$.  Consider the diagram below, which depicts boxes $(i, j-2), (i, j-1)$, and $(i,j)$ and the corresponding boxes $\widehat{(i,j-2)}, \widehat{(i,j-1)}, \widehat{(i,j)}$.

\[
\begin{array}{ccccc}
U & U \leftarrow b & U \leftarrow b \leftarrow a & U \leftarrow b  \leftarrow c & U \leftarrow b \leftarrow c \leftarrow a \\
\Tableau{ z & d & y} & \Tableau{ \tilde{z} & \tilde{d} & \tilde{y} } & \Tableau{ \hat{z} & \hat{d} & \tilde{y} } & \Tableau{\tilde{z} & c_{j-1}^i & \tilde{y} } & \Tableau{\tilde{z} & c_{j-1}^i & a_j^i}
\end{array}
\]

where either $\hat{z}$ or $\hat{d}$ could equal $a_j^i$ (but not both). We then get the inequalities $a_j^i \geq \tilde{y}$ which forces $a_j^i \geq \tilde{d}$. By Lemma \ref{scanning values} we know $b_{j-1}^i> a_j^i$, which implies $b_{j-1}^i > \tilde{d}\geq d$.  Thus the box $(i, j-1)$ is not in the path of $b$ and $\tilde{d}=d$.  On the other hand we see $\tilde{z} > c_{j-1}^i \geq b_{j-1}^i$, and since the path of $c$ cannot be strictly right of the path of $b$ we also see the box $(i, j-2)$ is not in the path of $b$ and thus $z=\tilde{z}$.  In the end we get the relations $z> b_{j-1}^i$ and $b_{j-1}^i> d$ which implies the box $(i, j-1)$ is in the path of $b$ and is a contradiction to the previously established condition on the box $(i, j-1)$.

Now we can assume $j-1=1$.  In this case the box $\widehat{(i, j-1)}=\widehat{(i,1)}$ is still in the path of $c$ and the position $\widehat{(i,j)}=\widehat{(i, 2)}$ is empty during the insertion of $a$.  With our assumption that $\widehat{(i,j)}=\widehat{(i,2)}$ is in the path of $a$ in $U \leftarrow b \leftarrow c \leftarrow a$ this implies the insertion of $b$ must have created a new box in the first column, say in position $(r, 1)$ with $r<i$.  This means position $(r, 2)$ is empty during the insertion of $a$ and by Lemma \ref{scanning values}, $b_j^r > a_j^i$ and thus $a_j^i$ must insert in position $(r,2)$. Which means $a$ cannot have $\widehat{(i,j)}=\widehat{(i,2)}$ in its path which is clearly a contradiction.

\emph{Step 2:} If $(i,j)$ is in the path of $a$ in $U \leftarrow b \leftarrow a$, then the box $\widehat{(i,j)}$ is in the path of $a$ in $U \leftarrow b \leftarrow c \leftarrow a$. As stated above, $j\neq 1$. Suppose for a contradiction that $\widehat{(i,j)}$ is not in the path of $a$. Then the box $\widehat{(i,j)}$ is in the path of $c$. As above consider the following diagram which depicts boxes $(i,j-1)$ and $(i,j)$ and the corresponding boxes $\widehat{(i,j-1)}$ and $\widehat{(i,j)}$.

\[
\begin{array}{ccccc}
U&U \leftarrow b & U \leftarrow b \leftarrow a & U \leftarrow b  \leftarrow c & U \leftarrow b \leftarrow c \leftarrow a \\
 \Tableau{ d&y}&\Tableau{ d & y } & \Tableau{ d & a_j^i } & \Tableau{ d & c_j^i } & \Tableau{d & c_j^i}
\end{array}
\]

where neither box $(i,j)$ nor $(i,j-1)$ can be in the path of $b$ since the box $(i,j)$ is in the path of $a$ and $\widehat{(i,j)}$ is in the path of $c$. From our assumptions we get the inequalities $d > a_j^i \geq y$ and $d > c_j^i>a_j^i$.  Now consider the scanning values obtained during the insertion of $b$.  Lemma \ref{scanning values} implies $b_j^i > a_{j+1}^i\geq a_j^i \geq y$ and $d > c_j^i \geq b_j^i$.  These inequalities force the box $(i,j)$ to be in the path of $b$, which contradicts properties previously established. This implies $a$ must have $\widehat{(i,j)}$ in its insertion path in $U \leftarrow b \leftarrow c \leftarrow a$.

\emph{Step 3:} The values bumped by $a_j^i$ is the same in both $U \leftarrow b \leftarrow a$ and $U\leftarrow b \leftarrow c\leftarrow a$.  If this were not the case, then both $a$ and $c$ would have the box $\widehat{(i,j)}$ in their respective paths, which violates Proposition \ref{bumping}.

\end{proof}

\begin{remark}{\label{knuth-remark}}
In addition to showing $U \leftarrow b \leftarrow c \leftarrow a = U \leftarrow b \leftarrow a \leftarrow c$, the proof above shows that the new boxes added by $a, b,$ and $c$ occupy the same corresponding positions in each of $U \leftarrow b \leftarrow c \leftarrow a$ and $U \leftarrow b \leftarrow a \leftarrow c$.
\end{remark}

\section{Littlewood Richardson Rule}{\label{sec:LRrule}}

In this section we state and prove the main result of this paper, which is

\begin{theorem}{\label{LRrule}}
Let $s_\lambda$ be the Schur function indexed by the partition $\lambda$, and let $\rsqschur_\alpha$ be the row-strict quasisymmetric Schur function indexed by the strong composition $\alpha$. We have
\begin{equation}{\label{LR-equation}}
\rsqschur_\alpha \cdot s_\lambda = \sum_{\beta} C_{\alpha, \lambda}^\beta \rsqschur_\beta
\end{equation}
where $C_{\alpha, \lambda}^\beta$ is the number of Littlewood-Richardson skew RCT of shape $\beta/\alpha$ and content $\reverse{\lambda}$.
\end{theorem}

\begin{proof}
It suffices to give a bijection $\rho$ between pairs $[U, T]$ and $[V,S]$ where $U$ is a RCT of shape $\alpha$, $T$ a tableau of shape $\lambda^t$, $V$ is a RCT of shape $\beta$, and $S$ is a LR skew RCT of shape $\beta/\alpha$ and content $\reverse{\lambda}$. Throughout this proof, $\lambda_1=m$.

Given a pair $[U,T]$, produce a pair $[V,S]=\rho([U,T])$ in the following way.  First use the classical RSK algorithm to produce a two-line array $A$ corresponding to the pair $(T, T_\lambda)$. Next, successively insert $\widecheck{A}$ into $U$ while simultaneously placing the entries of $\widehat{A}$ into the corresponding new boxes of a skew shape with original shape $\alpha/\alpha$. This clearly produces a RCT $V$ of some shape $\beta$ and a skew filling $S$ of shape $\beta/\gamma$ where $\strongof{\gamma}=\alpha$ and the content of $S$ is $\reverse{\lambda}$.

To show that the skew filling $S$ is indeed a LR skew RCT, first note that since $\widehat{A}$ is weakly decreasing, no row of $S$ will have any instance of entries that strictly increase when read left to right.  Since $A$ is a two-line array, if $i_r=i_s$ for $r \leq s$ then $j_r\leq j_s$.  Lemma \ref{scanning values} then implies that each row of $S$ has distinct entries. Thus, the rows of $S$ strictly decrease when read left to right.

Consider the portion of $A$ where $\widehat{A}$ takes the value $i$.  The corresponding entries in $\widecheck{A}$, when read from left to right, are the entries appearing in the $(m-i+1)$st column of $T$ read from bottom to top. Now consider a different portion of the two-line array $A$ where $\widehat{A}$ takes values $i$ and $i-1$. For the moment, let this portion of $A$ be denoted $A(i,i-1)$ and suppose the number of $i$'s is $r_i$ and the number of $(i-1)$'s is $r_{i-1}$, where $r_i\geq r_{i-1}$ since $\widehat{A}$ is a regular reverse lattice word.  We will let the Knuth transformation $\kone$ act on $A(i,i-1)$ by letting $\kone$ act on $\widecheck{A}(i,i-1)$ and by considering each vertical pair as a bi-letter. We will apply a sequence $\tau$ of transformations $\kone$ to $A(i,i-1)$ until $\tau[\widehat{A}(i,i-1)]$ consists of $r_i-r_{i-1}$ number of $i$'s followed by $r_{i-1}$ pairs of the form $(i,i-1)$. Such a sequence $\tau$ exists because the entry in row $r_{i-1}-k+1$ (for $1\leq k \leq r_{i-1}$) and column $m-i+2$ of $T$ is strictly less than each entry in column $m-i+1$ which appears weakly higher in $T$. If we replace $A(i,i-1)$ with $\tau[A(i,i-1)]$ in $A$ to obtain some array $B$, then Lemma \ref{knuth} and Remark \ref{knuth-remark} imply $$U \leftarrow \widecheck{B} = U \leftarrow \widecheck{A} =V,$$ and the corresponding new box created by any entry $j$ in $\widecheck{B}$ is in the same position as the new box created by the same entry $j$ in $\widecheck{A}$. The advantage of replacing $A$ with $B$ is that now Proposition \ref{bumping} can be applied to each of the $r_{i-1}$ pairs $(i,i-1)$ and their corresponding entries in $\widecheck{B}$ to imply that in any prefix of the column word of $S$, the number of $i$'s will be at least the number of $(i-1)$'s. Hence $w_{col}(S)$ is a regular reverse lattice word.

Next we check that each Type A and Type B triple in $S$ is an inversion triple.  Below are the eight possible configurations of Type A triples in the skew filling $S$.

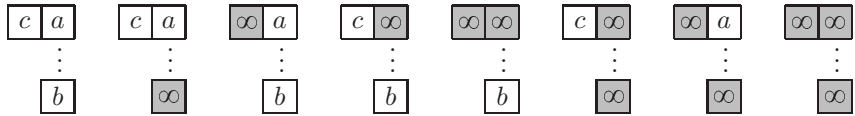
\begin{picture}(100,100)(-80,0)

\put(0, 50){\tableau{c&a}}
\put(26, 32){\vdots}
\put(18, 10){\tableau{b}}

\put(86, 32){\vdots}
\put(78, 10){\sqone}
\put(60, 50){\tableau{ c & a}}
\put(78, 10){\tableau{ \infty}}

\put(120, 50){\sqone}
\put(120,50){\tableau{\infty}}
\put(138, 50){\tableau{a}}
\put(146, 32){\vdots}
\put(138, 10){\tableau{b}}

\put(180, 50){\tableau{c}}
\put(198, 50){\sqone}
\put(198, 50){\tableau{\infty}}
\put(206, 32){\vdots}
\put(198, 10){\tableau{b}}

\put(240, 50){\sqtwo}
\put(240, 50){\tableau{\infty & \infty}}
\put(266, 32){\vdots}
\put(258, 10){\tableau{b}}

\put(300, 50){\tableau{c}}
\put(318, 50){\sqone}
\put(318, 50){\tableau{\infty}}
\put(326, 32){\vdots}
\put(318, 10){\sqone}
\put(318, 10){\tableau{\infty}}

\put(360, 50){\sqone}
\put(360, 50){\tableau{\infty}}
\put(378, 50){\tableau{a}}
\put(386, 32){\vdots}
\put(378, 10){\sqone}
\put(378, 10){\tableau{\infty}}

\put(420,50){\sqtwo}
\put(446, 32){\vdots}
\put(438, 10){\sqone}
\put(420,50){\tableau{ \infty & \infty}}
\put(438,10){\tableau{\infty}}

\end{picture}

For each arrangement in the figure above, the higher row is weakly longer than the lower row because we only consider Type A triples for now.  Note that the fourth and sixth arrangements cannot exist in $S$ by Definition \ref{RCT Insertion}, and the seventh arrangement cannot exist in $S$ by Lemma \ref{weakly longer}. We can check the remaining arrangements to prove each are inversion triples.  For the first arrangement $c$ must clearly be added before $a$ and Lemma \ref{weakly longer} implies $a$ is added before $b$, which forces the relation $c > a \geq b$. The second arrangement is always an inversion triple.  In the third arrangement Lemma \ref{weakly longer} implies $a$ must have been added before $b$, hence this arrangement is an inversion triple. The fifth and eighth arrangements are always inversion triples. 

Below are the eight possible arrangements of Type B triples in the skew filling $S$.

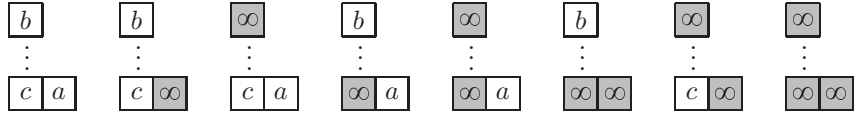
\begin{picture}(100,100)(-80,0)

\put(0, 50){\tableau{b}}
\put(8, 32){\vdots}
\put(0, 10){\tableau{c&a}}

\put(68, 32){\vdots}
\put(78, 10){\sqone}
\put(60, 50){\tableau{b}}
\put(78, 10){\tableau{ \infty}}
\put(60, 10){\tableau{c}}

\put(120, 50){\sqone}
\put(120,50){\tableau{\infty}}
\put(128, 32){\vdots}
\put(120, 10){\tableau{c&a}}

\put(180, 50){\tableau{b}}
\put(188, 32){\vdots}
\put(180, 10){\sqone}
\put(180, 10){\tableau{\infty}}
\put(198, 10){\tableau{a}}

\put(240, 50){\sqone}
\put(240, 50){\tableau{\infty}}
\put(248, 32){\vdots}
\put(240, 10){\sqone}
\put(240, 10){\tableau{\infty}}
\put(258, 10){\tableau{a}}

\put(300, 50){\tableau{b}}
\put(308, 32){\vdots}
\put(300, 10){\sqtwo}
\put(300, 10){\tableau{\infty&\infty}}

\put(360, 50){\sqone}
\put(360, 50){\tableau{\infty}}
\put(368, 32){\vdots}
\put(360, 10){\tableau{c}}
\put(378, 10){\sqone}
\put(378, 10){\tableau{\infty}}

\put(420,50){\sqone}
\put(420, 50){\tableau{\infty}}
\put(428, 32){\vdots}
\put(420, 10){\sqtwo}
\put(420,10){\tableau{ \infty & \infty}}

\end{picture}

In each of the arrangements above the higher row is strictly shorter than the lower row. Note that the second and seventh arrangement cannot exist in $S$ by Definition \ref{RCT Insertion}.  For the first arrangement Lemma \ref{lastrow} implies the boxes must have been added in the order $b$, $c$, $a$ or $c$, $a$, $b$, giving the relations $b \geq c > a$ or $c > a \geq b$. The third arrangement is always an inversion triple.  For the fourth arrangement, Lemma \ref{lastrow} implies $a$ must have been added before $b$, hence this arrangement is an inversion triple.  The fifth, sixth, and eighth arrangements are always inversion triples.

Thus, we have shown the skew filling $S$ is indeed a LR skew RCT of shape $\beta/\alpha$ and content $\reverse{\lambda}$.

Given a pair $[V,S]$, produce a pair $[U, T]=\rho^{-1}([V,S])$ in the following way.  We can un-insert entries from $V$ by using $S$ as a sort of road map. Specifically, un-insert the entry in $V$ whose box is in the same position as the first occurrence (in the column reading order) of the value $1$ in $S$. This produces a pair $(1, j)$ which is the last entry of a two-line array.  Next, proceed inductively by, at the $i$th step, un-inserting each entry of $V$ which corresponds to each occurrence of the value $i$ in $S$. The row-stirctness of $S$, combined with the triple conditions imposed on $S$, ensure that after each un-insertion from $V$ the resulting figure is an RCT.

What remains after un-inserting the entries is an RCT $U$ of shape $\alpha$ since $S$ had shape $\beta/\alpha$.  The two line array produced is a valid two-line array $A$ by virtue of $w_{col}(S)$ being a regular reverse lattice word.  By RSK, $A$ corresponds to a pair $(T, T_\lambda)$.  Thus we have a pair $[U, T]$ where $U$ is an RCT of shape $\alpha$ and $T$ is a tableau of shape $\lambda^t$.

\end{proof}

Figure \ref{fig-LRexample} gives an example of the bijection $\rho$ given in the proof of Theorem \ref{LRrule}.

\begin{figure}[ht]
\begin{center}
\[
\begin{array}{c}
\begin{pmatrix}
\;\tableau{  1 \\ 4 & 3 & 2 \\ 5 & 4 \\ 5 & 3 }  & , &
\tableau{  4 & 3 & 2 & 1 \\ 4 & 3 \\ 2} \; \vspace{6pt}\\ U & & V
\end{pmatrix} 
\qquad \stackrel{RSK}{\Leftrightarrow}
\vspace{12pt} \\
\begin{CD}
\begin{pmatrix} \;
\tableau{  1 \\ 4 & 3 & 2 \\ 5 & 4 \\ 5 & 3  } & , &
\begin{picture}(160,50)
$
\begin{pmatrix}
4 & 4 & 4 & 3 & 3 & 2 & 1 \\
2 & 4 & 4 & 3 & 3 & 2 & 1
\end{pmatrix} 
$
\end{picture} 
\vspace{6pt} \\  U & & (T, T_\lambda)
\end{pmatrix} @>\rho>>
\begin{pmatrix} \;
\tableau{
   1 \\
   3 & 2 \\
   4 & 3 & 2 \\
   4  \\ 
   5 & 4 & 3 & 2 & 1\\
   5 & 4 & 3   
} & , &
\begin{picture}(120,50)(230,70)
\put(245,72){\sqone}
\put(245,36){\sqthree}
\put(245,0){\sqtwo}
\put(245,-18){\sqtwo}
\put(222,72){\tableau{ \bas{\infty} &\infty \\
    \bas{\infty} & 4& 3  \\
     \bas{\infty} & \infty & \infty & \infty  \\ 
    \bas{\infty} & 4  \\
    \bas{\infty} & \infty & \infty &4&2&1 \\ 
    \bas{\infty} & \infty & \infty  & 3
}}
\end{picture}
\vspace{6pt} \\  V & & S
\end{pmatrix} 
\end{CD} 
\end{array}
\]
\caption{Example for a term of $\rsqschur_{(1,3,2,2)}\cdot s_{(3,2,1,1)}$}
\label{fig-LRexample}
\end{center}
\end{figure}

\subsection*{Acknowledgments} Thanks to Monica Vazirani and Sarah Mason for many helpful conversations.

\def\cprime{$'$}

\end{document}